\documentclass[11pt,reqno]{amsart}
\usepackage{amsmath}
\usepackage{amsfonts}
\usepackage{amssymb, amsthm, epsfig}
\usepackage{mathrsfs}
\usepackage{hyperref}
 \usepackage{yhmath}
\usepackage{cite}
\usepackage[utf8]{inputenc}
\usepackage[T1]{fontenc}
\usepackage{geometry}
\usepackage{graphicx}
\usepackage{verbatim}
\usepackage{color}

\theoremstyle{plain}
\newtheorem{thm}{Theorem}[section]

\newtheorem{lem}[thm]{Lemma}
\newtheorem{prop}[thm]{Proposition}

\theoremstyle{definition}

\theoremstyle{remark}
\newtheorem{rem}{Remark}[section]

\numberwithin{equation}{section}%{section}

\allowdisplaybreaks

\DeclareSymbolFont{lettersA}{U}{pxmia}{m}{it}
\DeclareMathSymbol{\piup}{\mathord}{lettersA}{"19}

\makeatletter

\newcommand{\Rmnum}[1]{\expandafter\@slowromancap\romannumeral#1@}
\makeatother
\begin{document}

\title[]{Stability of oblique sonic shocks in steady supersonic potential flow past a wedge}%
%\thanks{This work was partially supported by National Natural Science Foundation of China (12071278) and Natural Science Foundation of Shanghai (23ZR1422100).
%}%
\author[]{Geng Lai}

\dedicatory{Department of Mathematics, Shanghai University,
            Shanghai, 200444, P. R. China\\%}%
\baselineskip 18pt%
\tt E-mail: laigeng@shu.edu.cn}
%\subjclass{}%
\keywords{}%

%\date{}%

%\commby{}%
% ----------------------------------------------------------------
\begin{abstract}
When a uniform steady supersonic oncoming flow impinges on a straight wedge, if the wedge angle is smaller than the detachment angle, two types of steady oblique shocks satisfying the entropy condition will form in the flow field: weak shocks with supersonic or subsonic downstream flow, and strong shocks with subsonic downstream flow.
There have been many results on the stability of oblique shocks with subsonic or supersonic downstream flow. However, much less is known about the stability of oblique shocks with sonic downstream flow, since the flow behind the shock wave is very sensitive to disturbances.
This paper investigates the stability of oblique sonic shocks under the assumption of a straight wedge with non-uniform incoming flow.
 We reduce the problem to a degenerate hyperbolic free boundary problem.
 A local Lipschitz-continuous solution with sonic-supersonic downstream flow to the free boundary problem is constructed, using the classical Ascoli-Arzel\`{a} theorem and a diagonal argument.
The main difficulty of this free boundary problem lies in the hyperbolic degeneracy on the sonic line.
By adopting a weighted characteristic decomposition method, we establish a delicate estimate of the downstream flow near the sonic line.

\vskip 4pt
\noindent%
{Keywords. Oblique shock, sonic-supersonic flow, degenerate hyperbolic system, characteristic decomposition.}
\
\vskip 4pt
\noindent%
{2020 AMS subject classification.} Primary: 35L65; Secondary: 35L60, 35L67.
\end{abstract}

\maketitle

%\tableofcontents

\section{Introduction}
Two-dimensional (2D) steady supersonic flows past compressive and rarefactive ramps
are fundamental in gas dynamics, as they provide global or local structures in various flow fields.
As shown in Figure \ref{Figure1},
there is an infinitely long wedge consisting of a horizontal wall that is straight up to a sharp corner $\mathrm{O}$, followed by a straight ramp. A supersonic flow with a constant state arrives along the horizontal wall and turns at $\mathrm{O}$ into a new direction.
If the wedge angle is smaller than the detachment angle, two types of steady oblique shocks satisfying the entropy condition will form in the flow field: weak shocks with supersonic or subsonic downstream flow, and strong shocks with subsonic downstream flow.

%=============
\begin{figure}[htbp]
\begin{center}
\includegraphics[scale=0.45]{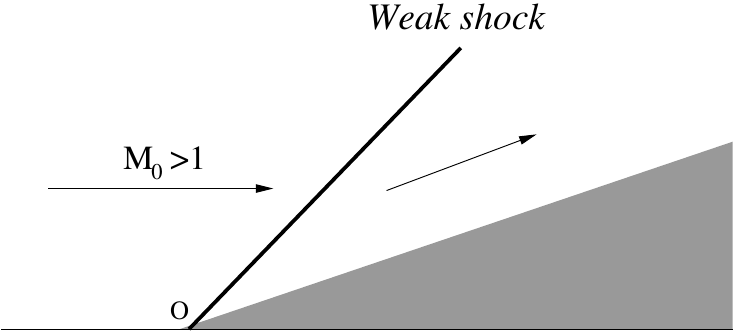}\qquad\qquad\qquad\qquad\includegraphics[scale=0.45]{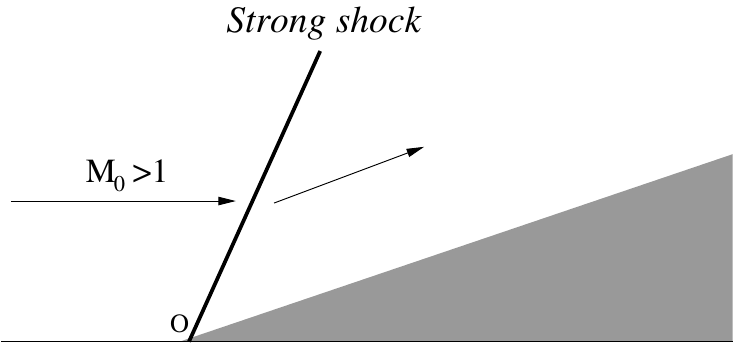}
\caption{\footnotesize 2D steady supersonic flow past a straight wedge.}
\label{Figure1}
\end{center}
\end{figure}
%==========

There have been many results on the stability of oblique shocks with subsonic or supersonic downstream flow.
For supersonic flows past a curved compressive ramp, the local existence of the corresponding curved shock with supersonic downstream flow (referred to as a supersonic shock) has been established in \cite{Gu,Li-Yu,Sch}. Yin \cite{Yin} applied a reflected characteristic method to obtain the global existence of a curved supersonic shock, provided that the curved ramp is a small perturbation of a straight one.
If the curved ramp is not a small perturbation of a straight one, the global existence of a curved supersonic shock has been established in \cite{HD,HQ}, provided that the Mach number of the incoming flow is sufficiently large.
The local existence of a curved conical supersonic shock for supersonic flow past a curved cone has been established in
\cite{CLD1,CLD2}.
For the global existence of curved conical supersonic shocks, we refer the reader to \cite{CXY,CY1,CY2,LWY1,HZ}. %,.
Chen-Fang \cite{CF} and Kim \cite{Eun} established the stability of 2D steady oblique transonic shocks with subsonic downstream flow. Xu and Ying \cite{XY1,XY2} established the global existence of curved conical transonic shocks.
Chen {\it et al.} \cite{CCX,CFf} also obtained the global existence of curved surface transonic shocks for supersonic flow past a three-dimensional compressive ramp.
Bae and Xiang \cite{BX} proved the existence of a detached shock solution for supersonic flow past a  blunt body.
We also refer to \cite{CKXZ,CKZ,CZZ,LL,WZ,Zhang} for the global existence of supersonic flows past Lipschitz-perturbed wedges and cones in the framework of BV functions.

If the downstream flow of an oblique shock is sonic, then it is very sensitive to disturbances. The flow downstream of the perturbed shock may be supersonic or subsonic, depending on how the shock is perturbed.
In one such case, the flow downstream of the shock transitions from sonic to subsonic, and the perturbed shock and its downstream flow can be determined by solving a degenerate elliptic boundary value problem. This case also frequently appears in 2D Riemann problems and 2D pseudo-steady shock reflection problems; see \cite{BCF1,BCF2,Canic1,Canic2,Canic3,CDX,ChenF1,ChenF2,Elling1,Elling2,Zheng1,Zheng2}. In another case, the flow downstream of the shock transitions from sonic to supersonic, and the problem reduces to solving a degenerate hyperbolic boundary value problem for the perturbed shock and its downstream flow.
Recently, there have been many results on the existence of 2D (pseudo-)steady sonic-supersonic flows; see \cite{HYB1,HYB2,HYB3,HYB4,LG1,LG2,Li4,WX1,WX2,Wen1,ZTY1}. These flow structures are mainly obtained by solving Cauchy or  initial–boundary value problems, and there are few results on free-boundary problems for shocks.
In this paper, we aim to construct a 2D steady sonic-supersonic flow arising from a free-boundary problem for a shock in potential flow.

%=============
\begin{figure}[htbp]
\begin{center}
\includegraphics[scale=0.42]{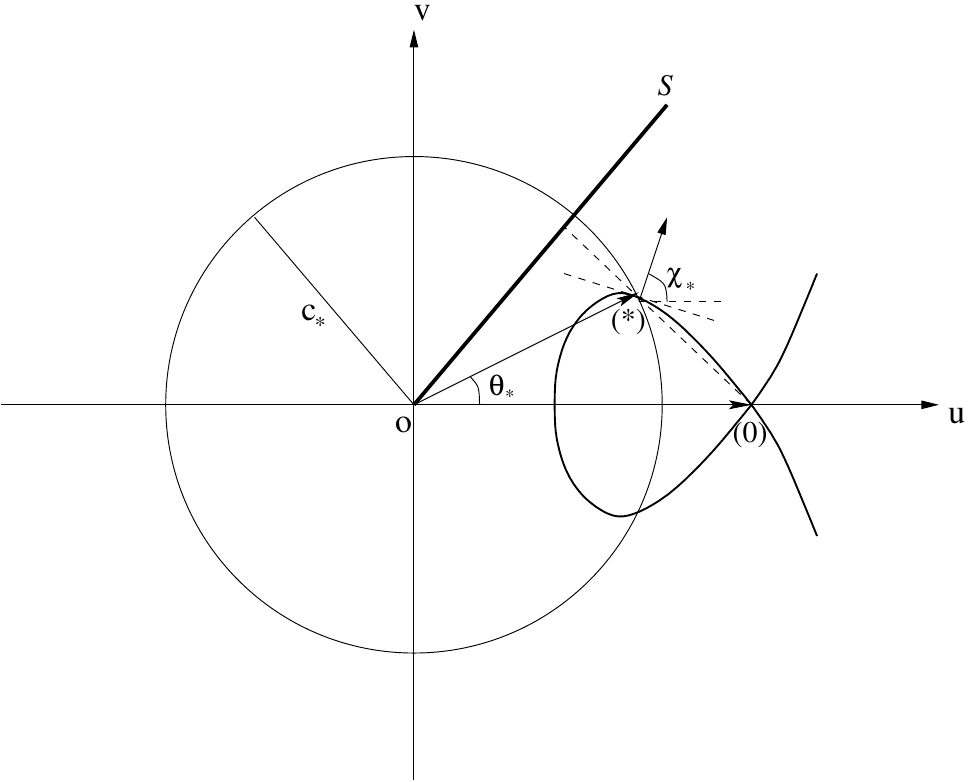}\quad\quad\includegraphics[scale=0.55]{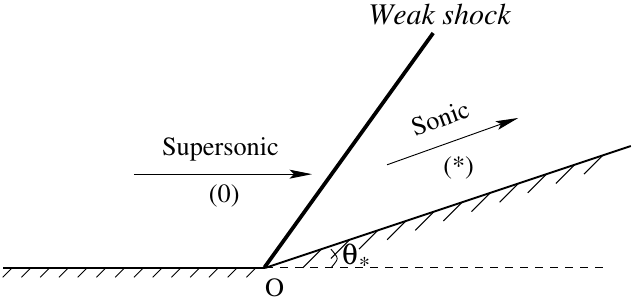}\qquad
\caption{\footnotesize Shock polar and an oblique sonic shock.}
\label{Figure2}
\end{center}
\end{figure}
%==========

The 2D steady potential flows of polytropic gases can be governed by
\begin{equation}\label{42501}
\left\{
  \begin{array}{ll}
 (\rho u)_x+(\rho v)_y=0, \\[4pt]
   u_y-v_x=0,
  \end{array}
\right.
\end{equation}
supplemented by Bernoulli's law
\begin{equation}\label{5802}
\frac{q^2}{2}+\frac{c^2}{\gamma-1}=\frac{\hat{q}^2}{2},
\end{equation}
where $(u, v)$ is the velocity, $\rho$ is the density, $c=\sqrt{\gamma \rho^{\gamma-1}}$ is the speed of sound, $\gamma$ is an adiabatic constant between $1$ and $3$,  $\hat{q}$ is a positive constant called the limiting speed, and $q=\sqrt{u^2+v^2}$.
From (\ref{5802}) we have
$$
c=\sqrt{\frac{\hat{q}^2-q^2}{\kappa}}\quad \mbox{and} \quad \rho=\left(\frac{\hat{q}^2-q^2}{\gamma\kappa}\right)^{\frac{1}{\gamma-1}},
$$
where $\kappa=\frac{2}{\gamma-1}$.

As shown in Figure \ref{Figure2} (right), a supersonic flow arrives with a constant velocity $(u_0, 0)$ and a constant density $\rho_0$  along
the horizontal wall and turns at $\mathrm{O}$ into a new direction through an oblique weak shock.
Assume that the inclination angle of the oblique shock is $\phi$. Then by the Rankine-Hugoniot conditions we know that the downstream state $(u, v)$ of the oblique shock satisfies
\begin{equation}\label{52401aa}
\left\{
  \begin{array}{ll}
    v\sin\phi+(u -u_0)\cos\phi=0, \\[4pt]
    (\rho  u -\rho_0u_0)\sin\phi-  \rho  v \cos\phi=0.
  \end{array}
\right.
\end{equation}
%where we still use $(u, v, \rho)$ to denote the flow on the back side of the shock.
Eliminating $\phi$ from (\ref{52401aa}), we get
\begin{equation}\label{82301}
 (\rho  u -\rho_0u_0)(u -u_0)+  \rho  v^2=0.
\end{equation}
 The curve in the $(u, v)$ plane given by (\ref{82301}) is the well-known shock polar; see Figure \ref{Figure2} (left). The branch $u^2+v^2<u_0^2$
 represents states behind shock fronts when the given state $(u_0, 0)$ refers to the front side state of the shock.
The shock polar intersects the circle $u^2+v^2=c_{*}^2$ at two points $(u_*, \pm v_*)$, where $c_*=\mu\hat{q}$ is the critical speed, $\mu=\sqrt{\frac{\gamma-1}{\gamma+1}}$, and $v_*>0$.
Let $\theta_*=\arctan(v_*/u_*)$. Then, if the wedge angle $\theta_w=\theta_*$, the downstream flow of the oblique weak shock is sonic, and the shock is usually called a sonic shock; if $0<\theta_w<\theta_*$ ($\theta_*<\theta_w<\theta_{ext}$, resp.), the downstream flow of the oblique weak shock is supersonic (subsonic, resp.), and the shock is sometimes called a supersonic (transonic, resp.) shock.

%In this paper, we study the stability of the oblique sonic shock.

%=============
\begin{figure}[htbp]
\begin{center}
\includegraphics[scale=0.45]{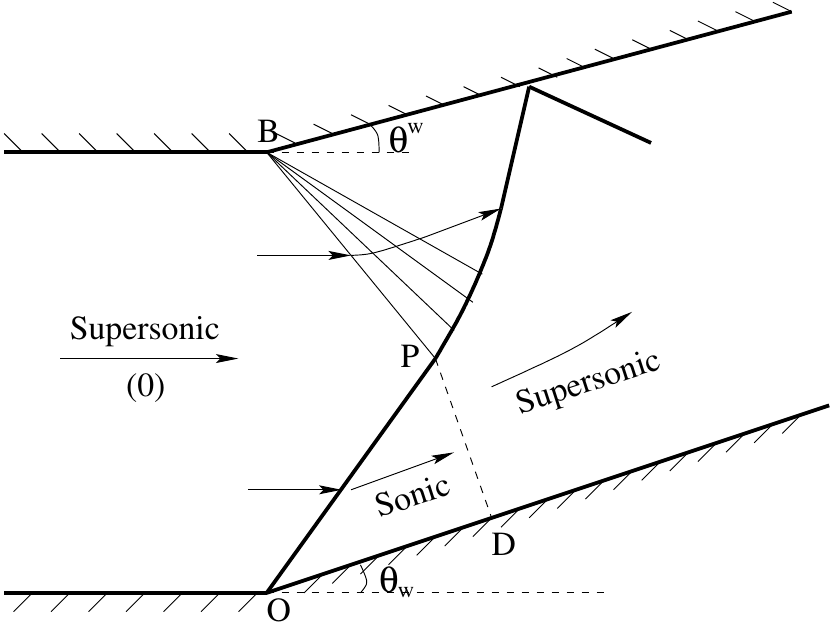}\qquad
\caption{\footnotesize Interaction of an oblique sonic shock with a centered expansion wave.}
\label{Figure3}
\end{center}
\end{figure}
%==========

An interesting question is what happens to the sonic shock when it is perturbed during its propagation.
To this end, we consider the following oblique-shock/expansion-fan interaction problem.
Consider Figure \ref{Figure3}, in which a supersonic flow with a constant velocity $(u_0, 0)$ and a constant density $\rho_0$ enters the duct from the left and encounters two consecutive wall deflections: a downward compressive corner at point $\mathrm{O}(0,0)$ on the lower wall, and an upward expansive corner at point $\mathrm{B}(0,1)$ on the upper wall.
We use $\theta_w$ and $\theta^w$ to denote the ramp angles of the lower and upper walls, respectively.
For the flow to negotiate this sudden turn, a centered expansion wave and an oblique shock wave are generated at the expansive and compressive corners, respectively.
The oblique shock wave intersects the leading characteristic of the expansion fan at point $\mathrm{P}(x_{_\mathrm{P}}, y_{_\mathrm{P}})$, initiating the interaction between these two waves.
The interaction between an oblique shock and an expansion fan is one of the fundamental wave interaction problems in steady supersonic flow, and it arises in a wide range of physical and engineering settings.
The interaction of an oblique supersonic shock with an expansion fan was studied by Li and Ben-dor \cite{LB}.
We assume that the ramp angle of the lower wall satisfies $\theta_w=\theta_*$. Then the oblique shock emanating from $\mathrm{O}$
is a sonic shock.
This paper studies the interaction of the oblique sonic shock and the centered expansion fan.

For smooth flow, system (\ref{42501}) can be written in the following matrix form
\begin{equation}
\left(
 \begin{array}{cc}
c^{2}-u^{2} & -uv\\
  0 & -1 \\
  \end{array}
  \right)\left(
           \begin{array}{c}
             u \\
             v \\
           \end{array}
         \right)_{x}+\left(
                         \begin{array}{cc}
                          -uv &  c^{2}-v^{2}\\
                           1 & 0 \\
                         \end{array}
                       \right)\left(
                                \begin{array}{c}
                                  u \\
                                  v \\
                                \end{array}
                              \right)_{y}~=~\left(
                                               \begin{array}{c}
                                                 0 \\
                                                 0 \\
                                               \end{array}
                                             \right).
                                             \label{matrix}
\end{equation}
%where $c=\sqrt{-\tau^2p'(\tau)}$ is the speed of sound.
The eigenvalues of (\ref{42501}) are determined
by
\begin{equation}
(v-\lambda u)^{2}-c^{2}(1+\lambda^{2})=0,\label{characteristice}
\end{equation}
which yields
\begin{equation}
\lambda=\lambda_{\pm}=\frac{uv\pm
c\sqrt{u^{2}+v^{2}-c^{2}}}{u^{2}-c^{2}}.
\end{equation}
So, if and only if $q>c$, system (\ref{42501}) is hyperbolic and has two families of wave characteristics $C_{\pm}$
defined as the integral curves of
$\frac{{\rm d}y}{{\rm d}x}=\lambda_{\pm}$.

%=============
%\begin{figure}[htbp]
%\begin{center}
%\includegraphics[scale=0.5]{characteristicd.pdf}
%\caption{ \footnotesize Characteristic angles and characteristic direction.}
%\label{Characteristicd}
%\end{center}
%\end{figure}
%==========

%We will use the method of characteristics,
%so we need the concept of the direction of the wave characteristics.
The directions of the wave
characteristics is defined as the tangent direction that forms an
acute angle $A$ with the direction of the flow velocity
$(u, v)$. The $C_+$
characteristic direction forms with the direction of the flow velocity
an angle $A$ from $(u, v)$ to $C_{+}$ in
the counterclockwise direction, and the $C_-$ characteristic direction forms with the
 direction of the flow velocity the angle $A$ from $(u, v)$ to $C_{-}$
in the clockwise direction.
The $C_{+}$ ($C_{-}$) characteristic angle is defined as the counterclockwise angle from the positive $x$-axis to the $C_{+}$ ($C_{-}$) characteristic direction.
We denote by $\alpha$ and
 $\beta$ the $C_{+}$ and $C_{-}$ characteristic angles,  respectively,
 where $0<\alpha-\beta<\pi$. Let $\sigma$ be the counterclockwise angle from the positive $x$-axis to the direction of the flow velocity.
 For the potential flow equations (\ref{42501}), we have
 \begin{equation}
\alpha=\sigma+A,\quad \beta=\sigma-A,\quad\sigma=\frac{\alpha+\beta}{2},\quad A=\frac{\alpha-\beta}{2}. \label{tau}
\end{equation}
Therefore, the relations between $(u,v,c)$ and $(\sigma, A, c)$ are
\begin{equation}\label{U}
u=\frac{c\cos\sigma}{\sin A}\quad \mbox{and}\quad v=\frac{c\sin\sigma}{\sin A}.
\end{equation}

Multiplying (\ref{matrix}) on the left by $(1,\mp c\sqrt{u^{2}+v^{2}-c^{2}})$, we get
\begin{equation}
\left\{
  \begin{array}{ll}
  \displaystyle \bar{\partial}_{+}u+\lambda_{-}\bar{\partial}_{+}v
 =0,  \\[4pt]
    \displaystyle  \bar{\partial}_{-}u+\lambda_{+}\bar{\partial}_{-}v=0,
  \end{array}
\right.\label{forma}
\end{equation}
where
\begin{equation}\label{pm}
\bar{\partial}_{+}:=\cos\alpha\partial_x+\sin\alpha\partial_y \quad \mbox{and}\quad \bar{\partial}_{-}:=\cos\beta\partial_x+\sin\beta\partial_y.
\end{equation}
\begin{lem}
Assume $(\bar{q}(\eta), \bar{\tau}(\eta), \bar{\sigma}(\eta))$ satisfies the equations
\begin{equation}\label{102302a}
\left\{
  \begin{array}{ll}
    \bar{q}'(\eta)\cos(\bar{A}(\eta))-\bar{q}(\eta)\bar{\sigma}'(\eta)\sin(\bar{A}(\eta))=0, \\[4pt]
    \bar{q}(\eta)\bar{q}'(\eta)+\bar{\tau}(\eta)p'(\bar{\tau}(\eta))\bar{\tau}'(\eta)=0,  \\[4pt]
    \bar{\sigma}(\eta)-\bar{A}(\eta)=\eta
  \end{array}
\right.
\end{equation}
 for $\eta\in (\eta', \eta'')$,
where $\bar{q}(\eta)>\bar{c}(\eta)>0$,
\begin{equation}\label{102303a}
\bar{A}(\eta)=\arcsin\Big(\frac{\bar{c}(\eta)}{\bar{q}(\eta)}\Big)\quad \mbox{and}\quad \bar{c}(\eta)=\bar{\tau}(\eta)\sqrt{-p'(\bar{\tau}(\eta))}.
\end{equation}
Let
\begin{equation}\label{102306a}
\bar{u}(\eta)=\bar{q}(\eta)\cos(\bar{\sigma}(\eta))\quad \mbox{and}\quad \bar{v}(\eta)=\bar{q}(\eta)\sin(\bar{\sigma}(\eta)).
\end{equation}
Then the function
\begin{equation}
(u, v)=(\bar{u}, \bar{v})(\eta), \quad   \eta'<\eta<\eta''
\end{equation}
is a centered simple wave solution with straight $C_{-}$ characteristic lines of  (\ref{42501}) on the fan-shaped domain
\begin{equation}\label{102502}
\Delta=\big\{(x, y)~\big|~(x, y)=(r\cos\eta, r\sin\eta+1),~ r>0,
~\eta\in(\eta', \eta'')
\big\}.
\end{equation}
\end{lem}
\begin{proof}
First, by the second equation (\ref{102302a}) we know that Bernoulli's law (\ref{5802}) holds.
From the third equation of (\ref{102302a}) we have
\begin{equation}\label{82104}
\big(\bar{u}(\eta), \bar{v}(\eta)\big)\cdot(-\sin\eta, \cos\eta)=\bar{c}(\eta).
\end{equation}
So, for any fixed $\eta\in (\eta', \eta'')$, the ray $x=r\cos\eta$, $y=r\sin\eta+1$ ($r>0$)
is tangent to the sonic circle $(x-\bar{u}(\eta))^2+(y-\bar{v}(\eta))^2=\bar{c}^2(\eta)$.
Therefore, by the result of Section 2 we see that this ray
 is
a straight $C_{+}$ characteristic line. Hence, we have $$\beta=\eta, \quad \bar{\partial}_{-}u=0,\quad \bar{\partial}_{-}v=0,\quad\mbox{and} \quad\bar{\partial}_{-}\tau=0\quad \mbox{in}~\Delta.$$ Consequently,
\begin{equation}\label{82101}
\bar{\partial}_{-}u+\lambda_{+}\bar{\partial}_{-}v=0\quad\mbox{in}~ \Delta.
\end{equation}

By the first equation of (\ref{102302a}) and $\beta=\eta$, we have
\begin{equation}\label{82102}
\begin{aligned}
\bar{\partial}_{+}u+\lambda_{-}\bar{\partial}_{+}v
&=\big( \bar{u}'(\eta)+\tan\eta  \bar{v}'(\eta)\big)\bar{\partial}_{-}\eta\\&=\big(\bar{q}'(\eta)\cos(\bar{A}(\eta))
-\bar{q}(\eta)\bar{\sigma}'(\eta)\sin(\bar{A}(\eta))\big)\frac{\bar{\partial}_{-}\eta}{\cos\eta}=0 \quad\mbox{in}~ \Delta.
\end{aligned}
\end{equation}
This completes the proof of this lemma.
\end{proof}

The system (\ref{102302a}) can by (\ref{102303a}) be written as
\begin{equation}\label{102304a}
\left\{
  \begin{array}{ll}
    \displaystyle\bar{q}'=-\frac{2\bar{c}p'(\bar{\tau})\cos \bar{A}}{\bar{\tau}p''(\bar{\tau})},\\[10pt]
     \displaystyle\bar{\tau}'=\frac{2\bar{q}\bar{c}\cos \bar{A}}{\bar{\tau}^2p''(\bar{\tau})},\\[10pt]
 \displaystyle\bar{\sigma}'=-\frac{2 p'(\bar{\tau})\cos^2 \bar{A}}{\bar{\tau}p''(\bar{\tau})}.
  \end{array}
\right.
\end{equation}
We consider (\ref{102304a}) with the initial data
\begin{equation}\label{102304b}
(\bar{q}, \bar{\tau}, \bar{\sigma})(\eta_0)=(u_0, \tau_0, 0),
\end{equation}
where $\eta_0=-A_0$, $\tau_0=1/\rho_0$, $A_0=\arcsin(c_0/u_0)$, and $c_0=\sqrt{\gamma \rho_0^{\gamma-1}}$.
Then there exists a $\eta_1>\eta_0$ such that the problem (\ref{102304a}), (\ref{102304b}) admits a solution on $[\eta_0, \eta_1]$. Moreover, the solution satisfies $\bar{\sigma}(\eta_1)=\theta^{w}$.
Then the centered expansion fan can be determined by (\ref{102306a}) with $\eta=\arctan(\frac{y-1}{x})$.

%From the second equation of (\ref{102304a}) we also have
%\begin{equation}
%\tilde{c}'=-\frac{(2p'(\tilde{\tau})+\tilde{\tau}p''(\tilde{\tau}))\tilde{q}\cos \tilde{A}}{\tilde{\tau}p''(\tilde{\tau})}.
%\end{equation}

From the point $\mathrm{P}$, we draw a straight ray with the inclination angle $\theta_*-\frac{\pi}{2}$. This ray intersects the lower wall at a point $\mathrm{D}$. %Then the flow in the triangle domain $\mathrm{OPD}$ is the constant state $(u_*, v_*)$.
To study the interaction of the oblique sonic shock and the centered expansion wave, we consider (\ref{42501}) with the following boundary conditions:
\begin{equation}\label{b1}
(u, v)=(u_*, v_*)\quad \mbox{on}\quad \overline{\mathrm{PD}};
\end{equation}
\begin{equation}\label{b2}
v(x,y)=u(x,y)\tan\theta_{*}\quad \mbox{on}\quad y=x\tan\theta_*, \quad x>x_{_\mathrm{D}};
\end{equation}
\begin{equation}\label{b3}
(\rho  u -\rho_fu_f)(u -u_f)+  (\rho  v -\rho_fv_f)(v -v_f)=0\quad \mbox{on}\quad y=\psi(x), \quad x>x_{_\mathrm{P}};
\end{equation}
\begin{equation}\label{b4}
\left\{
  \begin{array}{ll}
   \displaystyle \frac{{\rm d}\psi(x)}{{\rm d}x}=-\frac{u(x, \psi(x))-u_f(x, \psi(x))}{v(x, \psi(x))-v_f(x, \psi(x))}, & \hbox{$x>x_{_\mathrm{P}}$,} \\
    \psi(x_{_\mathrm{P}})=y_{_\mathrm{P}},
  \end{array}
\right.
\end{equation}
where $(u_f, v_f)(x, y)=(\bar{u}, \bar{v})(\eta)$ and $(\bar{u}, \bar{v})(\eta)$ is the solution of the problem (\ref{102304a}), (\ref{102304b}).

The main result is stated as follows.
\begin{thm}\label{main}
There exists a small $\delta>0$ such that the free boundary problem (\ref{42501}), (\ref{b1})--(\ref{b4}) admits a Lipschitz-continuous solution on a  domain $\Sigma(\delta)$ bounded by $\overline{\mathrm{PD}}$, $y=\psi(x)$ ($x_{_\mathrm{P}}<x<x_{_\mathrm{P}}+\delta$), the lower wall $y=x\tan\theta_*$, and a forward $C_{-}$ characteristic line issued from the point $(x_{_\mathrm{P}}+\delta, \psi(x_{_\mathrm{P}}+\delta))$. Moreover, the solution satisfies $q>c_{*}$ in $\overline{\Sigma(\delta)}\setminus\overline{\mathrm{PD}}$.
\end{thm}

The main difficulty for the existence of a local solution is that the system is degenerate on the sonic line $\overline{\mathrm{PD}}$. Therefore, the result by Li and Yu \cite{Li-Yu} on the typical free boundary  problem of quasilinear hyperbolic systems does not apply here. To overcome the difficulty caused by the degeneracy, we consider a regularized problem by assuming that the lower wall angle satisfies $\theta_w=\theta_*-\varepsilon$, where $\varepsilon$ is a small positive constant.
Using the characteristic decomposition method, we establish a uniform $C^{0,1}$ estimate with respect to $\varepsilon$ for the solution of the regularized problem.

\section{Characteristic equations and decompositions}
\subsection{Characteristic equations}

From (\ref{U}) we have
\begin{equation}
\bar{\partial}_{\pm}u=\frac{\cos\sigma}{\sin A}\bar{\partial}_{\pm}c+\frac{c\cos\alpha\bar{\partial}_{\pm}\beta
-c\cos\beta\bar{\partial}_{\pm}\alpha}{2\sin^{2}A},\label{1a}
\end{equation}
\begin{equation}
\bar{\partial}_{\pm}v=\frac{\sin\sigma}{\sin A}\bar{\partial}_{\pm}c
+\frac{c\sin\alpha\bar{\partial}_{\pm}\beta
-c\sin\beta\bar{\partial}_{\pm}\alpha}{2\sin^{2}A}.\label{2a}
\end{equation}
Inserting (\ref{1a}) and (\ref{2a}) into (\ref{forma}), we obtain
\begin{equation}\label{3a}
\bar{\partial}_{+}c=\frac{c}{\sin2A}
(\bar{\partial}_{+}\alpha-\cos2A\bar{\partial}_{+}\beta)
\quad\mbox{and}\quad
\bar{\partial}_{-}c=\frac{c}{\sin2A}
(\cos2A\bar{\partial}_{-}\alpha-\bar{\partial}_{-}\beta).
\end{equation}

Differentiating Bernoulli's law (\ref{5802}) and
using (\ref{1a}) and (\ref{2a}), we get
\begin{equation}
\left(\frac{1}{\sin^{2}A}+\kappa\right)\bar{\partial}_{\pm}c=
\frac{c\cos A}{2\sin^{3}A}(\bar{\partial}_{\pm}\alpha-\bar{\partial}_{\pm}\beta).\label{bB1a}
\end{equation}
%where $\kappa=\frac{2}{\gamma-1}$.

Inserting (\ref{bB1a}) into (\ref{3a}), we obtain
\begin{equation}\label{6a}
\bar{\partial}_{+}\alpha=\Gamma\cos^{2}A\bar{\partial}_{+}\beta\quad \mbox{and}\quad \bar{\partial}_{-}\beta=\Gamma\cos^{2}A\bar{\partial}_{-}\alpha,
\end{equation}
where
$\Gamma=\frac{3-\gamma}{\gamma+1}-\tan^2A$.

Combining (\ref{3a}) and (\ref{6a}), we have
\begin{equation}
c\bar{\partial}_{+}\beta=-(1+\kappa)\tan A\bar{\partial}_{+}c,\label{7a}
\end{equation}
\begin{equation}
c\bar{\partial}_{+}\alpha=-\left(\frac{1+\kappa}{2}\right)\Gamma\sin2 A\bar{\partial}_{+}c,\label{10a}
\end{equation}
\begin{equation}\label{82301a}
c\bar{\partial}_{-}\alpha=(1+\kappa)\tan A\bar{\partial}_{-}c,
\end{equation}
\begin{equation}
c\bar{\partial}_{-}\beta=\left(\frac{1+\kappa}{2}\right) \Gamma\sin2 A\bar{\partial}_{-}c.\label{8a}
\end{equation}

From (\ref{1a}), (\ref{2a}), and (\ref{7a})--(\ref{8a}) we have
\begin{equation}\label{11a}
\bar{\partial}_{+}u=\kappa\sin\beta\bar{\partial}_{+}c,
\quad\bar{\partial}_{-}u=-\kappa\sin\alpha\bar{\partial}_{-}c,
\end{equation}
\begin{equation}\label{72804a}
\bar{\partial}_{+}v=-\kappa\cos\beta\bar{\partial}_{+}c,
\quad\bar{\partial}_{-}v=\kappa\cos\alpha\bar{\partial}_{-}c.
\end{equation}

\begin{rem}
From (\ref{11a}) and (\ref{72804a}) we can see that the bounds of $|\nabla u|$ and $|\nabla v|$ can be controlled by the bound of $|\nabla c|$.
\end{rem}

From Bernoulli's law (\ref{5802}) we also have
\begin{equation}\label{pma}
\bar{\partial}_{\pm}c= \frac{c\cot A}{1+\kappa \sin^2 A}\bar{\partial}_{\pm}A.
\end{equation}

\subsection{Riemann invariants}
The Riemann invariants of (\ref{42501}) are defined as
\begin{equation}\label{Riemanni}
r_{\pm}(\sigma,q)=\sigma\pm\int^{q}\frac{\sqrt{q^{2}-c^{2}}}{qc}dq.
\end{equation}
%in which $c=\ddot{c}(q)$.
In view of the Riemann invariants, we have
\begin{equation}\label{52801}
\left\{
  \begin{array}{ll}
    \bar{\partial}_{+}r_{-}=0,  \\[2pt]
     \bar{\partial}_{-}r_{+}=0.
  \end{array}
\right.
\end{equation}

It follows from
\begin{equation}
\left(
  \begin{array}{cc}
    \displaystyle\frac{\partial r_{+}}{\partial \sigma} & \displaystyle\frac{\partial r_{+}}{\partial q}   \\[8pt]
    \displaystyle\frac{\partial r_{-}}{\partial \sigma}  & \displaystyle\frac{\partial r_{-}}{\partial q}
  \end{array}
\right)\left(
         \begin{array}{cc}
          \displaystyle \frac{\partial \sigma}{\partial r_{+}}  & \displaystyle\frac{\partial \sigma}{\partial r_{-}} \\[8pt]
           \displaystyle\frac{\partial q}{\partial r_{+}} & \displaystyle \frac{\partial q}{\partial r_{-}}
         \end{array}
       \right)=\left(
                 \begin{array}{cc}
                   1 & 0 \\[12pt]
                   0 & 1 \\
                 \end{array}
               \right)
\end{equation}
that
\begin{equation}
\frac{\partial\sigma}{\partial r_{\pm}}=\frac{1}{2}\quad \mbox{and} \quad
\frac{\partial q}{\partial
r_{\pm}}=\pm \frac{qc}{2\sqrt{q^{2}-c^{2}}}=\pm \frac{q\sin A}{2\cos A}.\label{1201}
\end{equation}

Thus, by (\ref{U}) we have
\begin{equation}\label{6603}
\frac{\partial u}{\partial r_{+}}=-\frac{q\sin\beta}{2\cos A}, \quad \frac{\partial u}{\partial r_{-}}=-\frac{q\sin\alpha}{2\cos A}, \quad \frac{\partial v}{\partial r_{+}}=\frac{q\cos\beta}{2\cos A}, \quad \mbox{and} \quad
\frac{\partial v}{\partial r_{-}}=\frac{q\cos\alpha}{2\cos A}.
\end{equation}
Combining these with Bernoulli's law (\ref{5802}) we also have
\begin{equation}\label{6606}
\frac{\partial \rho}{\partial r_{-}}=\frac{\rho }{\sin (2A)}\quad \mbox{and}\quad \frac{\partial \rho}{\partial r_{+}}=-\frac{\rho }{\sin (2A)}.
\end{equation}

\subsection{Characteristic decompositions}

\begin{prop}
(Commutator relation)
\begin{equation}\label{comm}
\begin{array}{rcl}
\bar{\partial}_{-} \bar{\partial}_{+}- \bar{\partial}_{+} \bar{\partial}_{-}=
\displaystyle\frac{1}{\sin2 A}\Big[\big(\cos2 A \bar{\partial}_{+}\beta- \bar{\partial}_{-}\alpha\big) \bar{\partial}_{-}-
\big( \bar{\partial}_{+}\beta-\cos2 A \bar{\partial}_{-}\alpha\big) \bar{\partial}_{+}\Big].
\end{array}
\end{equation}
\end{prop}
\begin{proof}
This commutator relation was first derived by Li, Zhang, and Zheng \cite{Li-Zhang-Zheng}. We omit the details.
\end{proof}

\begin{prop}
For the variable $c$,
we have the characteristic decompositions
\begin{equation}\label{cd}
\left\{
  \begin{array}{ll}
 \displaystyle c\bar{\partial}_{-}\bar{\partial}_{+}c~=~
   \frac{(\gamma+1)\bar{\partial}_{+}c\bar{\partial}_{+}c}{2(\gamma-1)\cos^2A}+\frac{(\gamma+1)-2\sin^2 2A}{2(\gamma-1)\cos^2 A}
   \bar{\partial}_{-}c\bar{\partial}_{+}c,
\\[14pt]
   \displaystyle c\bar{\partial}_{+}\bar{\partial}_{-}c~=~
\frac{(\gamma+1)\bar{\partial}_{-}c \bar{\partial}_{-}c}{2(\gamma-1)\cos^2A}+\frac{(\gamma+1)-2\sin^2 2A}{2(\gamma-1)\cos^2 A}
   \bar{\partial}_{+}c \bar{\partial}_{-}c.
  \end{array}
\right.
\end{equation}
\end{prop}
\begin{proof}
Using the commutator relation (\ref{comm}) for the variable $u$, we have
\begin{equation}\label{4203}
\begin{array}{rcl}
\bar{\partial}_{-} \bar{\partial}_{+}u- \bar{\partial}_{+} \bar{\partial}_{-}u=
\displaystyle\frac{1}{\sin2 A}\Big[\big(\cos2 A \bar{\partial}_{+}\beta- \bar{\partial}_{-}\alpha\big) \bar{\partial}_{-}u-
\big( \bar{\partial}_{+}\beta-\cos2 A \bar{\partial}_{-}\alpha\big) \bar{\partial}_{+}u\Big].
\end{array}
\end{equation}
Inserting (\ref{11a}) into (\ref{4203})
and using the commutator relation (\ref{comm}) for the variable $c$, we can get (\ref{cd}).
We refer the reader to \cite{LG3} for the details.
\end{proof}

\begin{lem}\label{simple}
The flow in a region adjacent to a straight characteristic line is a simple wave (see \cite{CaF}).
\end{lem}

From (\ref{pma}) and (\ref{cd}), we have the following weighted characteristic decompositions
\begin{equation}\label{wcd}
\left\{
  \begin{array}{ll}
 \displaystyle c\bar{\partial}_{-}\Big(\frac{\bar{\partial}_{+}c}{\cos^2 A}\Big)~=~
\mathcal{F} \frac{\bar{\partial}_{+}c}{\cos^2 A}\cdot\frac{\bar{\partial}_{-}c}{\cos^2 A}+\frac{\gamma+1}{2(\gamma-1)}\Big(\frac{\bar{\partial}_{+}c}{\cos^2 A}\Big)^2,
\\[16pt]
   \displaystyle c\bar{\partial}_{+}\Big(\frac{\bar{\partial}_{-}c}{\cos^2 A}\Big)~=~
\mathcal{F}
\frac{\bar{\partial}_{+}c}{\cos^2 A}\cdot\frac{\bar{\partial}_{-}c}{\cos^2 A}+\frac{\gamma+1}{2(\gamma-1)}\Big(\frac{\bar{\partial}_{-}c}{\cos^2 A}\Big)^2,
  \end{array}
\right.
\end{equation}
where
$$
\mathcal{F}=\frac{(\gamma+1)+(2\gamma-14)\sin^2 A\cos^2 A+2(\gamma+1)\sin^2A(1+\sin^2A)}{2(\gamma-1)}.
$$

%Using the commutator relation for $\bar{\partial}_{+}c$ and $\bar{\partial}_{-}c$ and recalling (\ref{cd}), we also have the characteristic decompositions for higher order derivatives
%\begin{equation}\label{cd2}
%\left\{
%  \begin{array}{ll}
%  \begin{aligned}
% \displaystyle c\bar{\partial}_{-}(\bar{\partial}_{+}\bar{\partial}_{+}c)~=~&
%\left(\frac{3(\gamma+1)\bar{\partial}_{+}c}{2(\gamma-1)\cos^2A}+
%\frac{(\gamma+1-4\sin^2A)\bar{\partial}_{-}c}{\gamma-1}\right)\bar{\partial}_{+}\bar{\partial}_{+}c\\[4pt]
%&\qquad ~+~\mathcal{Q}_{1}\frac{(\bar{\partial}_{+}c)^3}{c\cos^4 A}+\mathcal{Q}_{2}\frac{(\bar{\partial}_{+}c)^2\bar{\partial}_{-}c}{c\cos^4 A},
%\end{aligned}
%\\[36pt]
%\begin{aligned}
%   \displaystyle c\bar{\partial}_{+}(\bar{\partial}_{-}\bar{\partial}_{-}c)~=~&
%\left(\frac{3(\gamma+1)\bar{\partial}_{-}c}{2(\gamma-1)\cos^2A}+
%\frac{(\gamma+1-4\sin^2A)\bar{\partial}_{+}c}{\gamma-1}\right)\bar{\partial}_{-}\bar{\partial}_{-}c~\\[4pt]
%&\qquad~+~\mathcal{Q}_{1}\frac{(\bar{\partial}_{-}c)^3}{c\cos^4 A}+\mathcal{Q}_{2}\frac{(\bar{\partial}_{-}c)^2\bar{\partial}_{+}c}{c\cos^4 A},
%\end{aligned}
%  \end{array}
%\right.
%\end{equation}
%where $\mathcal{Q}_{i}=a_{i1}+a_{i2}\sin^2 A+ a_{i3}\sin^4A$ ($i=1, 2, 3$), and $a_{ij}$ ($i, j=1, 2, 3$) are constants.

%We shall use (\ref{cd}) and (\ref{cd2}) to estimate the second order derivatives of $c$.

\section{Interaction of an oblique sonic shock with a centered simple wave}
In order to overcome the difficulty caused by the hyperbolic degeneracy on the sonic line, we consider a regularized problem by assuming that the ramp angle of the lower wall satisfies $\theta_w=\theta_*-\varepsilon$, with $\varepsilon$ being an arbitrary small positive constant.
Under this assumption, there is an oblique shock emanating from $\mathrm{O}$, with a downstream supersonic flow state $(u_1, v_1)$.
As shown in Figure \ref{Figure4}, the leading characteristic line of the centered expansion fan meets the oblique shock at a point $\mathrm{P}_{\varepsilon}$.
From the point $\mathrm{P}_{\varepsilon}(x_{_{\mathrm{P}_{\varepsilon}}},y_{_{\mathrm{P}_{\varepsilon}}})$, we draw a forward straight $C_{-}$ characteristic line of the constant state
$(u_1, v_1)$, and this characteristic line intersects the ramp at a point $\mathrm{D}_{\varepsilon}(x_{_{\mathrm{D}_{\varepsilon}}},y_{_{\mathrm{D}_{\varepsilon}}})$.
To study the interaction of the oblique supersonic shock with the centered expansion wave, we consider (\ref{42501}) with the following boundary conditions:
\begin{equation}\label{br1}
v(x,y)=u(x,y)\tan\theta_w\quad \mbox{on}\quad y=x\tan\theta_w, \quad x>x_{_{\mathrm{D}_{\varepsilon}}};
\end{equation}
\begin{equation}\label{br2}
(u, v)=(u_1, v_1)\quad \mbox{on}\quad \overline{\mathrm{P_{\varepsilon}D_{\varepsilon}}};
\end{equation}
\begin{equation}\label{br3}
(\rho  u -\rho_fu_f)(u -u_f)+  (\rho  v -\rho_fv_f)(v -v_f)=0\quad \mbox{on}\quad y=\psi_{\varepsilon}(x), \quad x>x_{_{\mathrm{P}_{\varepsilon}}};
\end{equation}
\begin{equation}\label{br4}
\left\{
  \begin{array}{ll}
   \displaystyle \frac{{\rm d}\psi_{\varepsilon}(x)}{{\rm d}x}=-\frac{u(x, \psi_{\varepsilon}(x))-u_f(x, \psi_{\varepsilon}(x))}{v(x, \psi_{\varepsilon}(x))-v_f(x, \psi_{\varepsilon}(x))}, & \hbox{$x>x_{_{\mathrm{P}_{\varepsilon}}}$,} \\
    \psi_{\varepsilon}(x_{_{\mathrm{P}_{\varepsilon}}})=y_{_{\mathrm{P}_{\varepsilon}}},
  \end{array}
\right.
\end{equation}
where $(u_f, v_f)(x, y)=(\bar{u}, \bar{v})(\eta)$ and $\eta=\arctan(\frac{y-1}{x})$.
Here, the unknown function $y=\psi_{\varepsilon}(x)$ represents the shock propagating from the point $\mathrm{P}_{{\varepsilon}}$.
%=============
\begin{figure}[htbp]
\begin{center}
\includegraphics[scale=0.46]{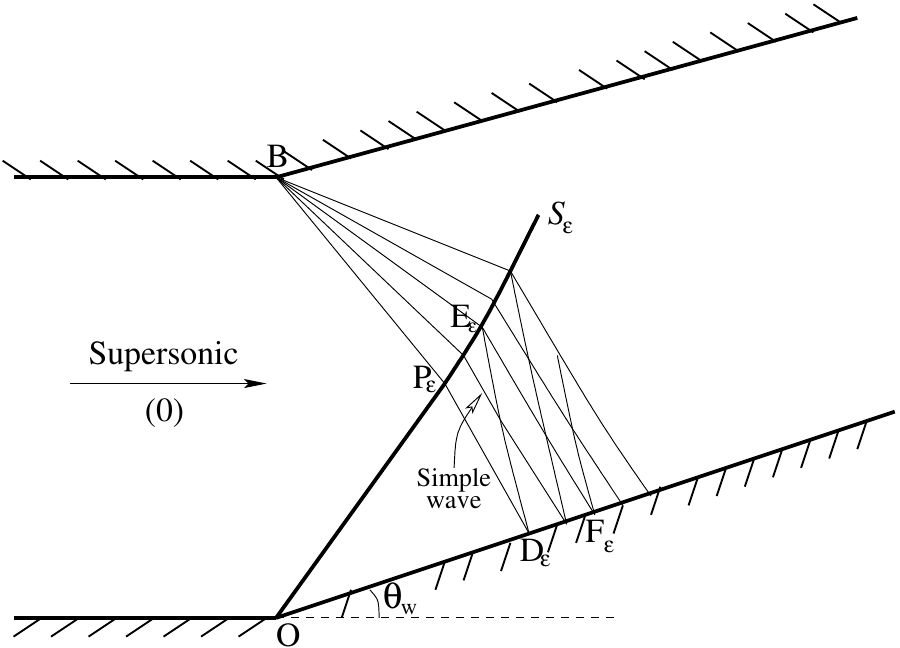}
\caption{\footnotesize Interaction of an oblique supersonic shock with a centered expansion fan wave.}
\label{Figure4}
\end{center}
\end{figure}
%==========

%\subsection{Estimates along the shock wave}
%The oblique shock will continue to propagate from the point $P_{\varepsilon}$ and become curved.
For convenience, we use $S_{\varepsilon}$ to denote the curved shock propagating from the point $\mathrm{P}_{{\varepsilon}}$.
Let $\phi(x, \psi(x))=\arctan\psi_{\varepsilon}'(x)$. Then by (\ref{br3}) and (\ref{br4}),
along the shock ${\it S}_{\varepsilon}$, we have
\begin{equation}\label{52401a}
\left\{
  \begin{array}{ll}
    (v -v_f)\sin\phi+(u -u_f)\cos\phi=0, \\[4pt]
    (\rho  u -\rho_fu_f)\sin\phi-  (\rho  v -\rho_fv_f)\cos\phi=0.
  \end{array}
\right.
\end{equation}
%where we still use $(u, v, \rho)$ to denote the downstream flow of the shock.
Eliminating $\phi$ from (\ref{52401a}), we get
\begin{equation}\label{Polar}
G(u_f,v_f, u, v):= (\rho  u -\rho_fu_f)(u -u_f)+  (\rho  v -\rho_fv_f)(v -v_f)=0\quad \mbox{along}\quad S_{\varepsilon}.
\end{equation}

 We denote by $N$ and $L$
the components of $(u, v)$ normal and tangential to the shock direction, respectively.
Then we have
\begin{equation}\label{62604}
L=u\cos \phi+v\sin \phi, \quad N=u\sin \phi- v\cos \phi,
\end{equation}
\begin{equation}\label{52304}
u=N\sin\phi+L\cos\phi\quad \mbox{and}\quad v=L\sin\phi-N\cos\phi.
\end{equation}
The Rankine-Hugoniot conditions (\ref{52401a}) can be written as
\begin{equation}\label{RH}
\left\{
  \begin{array}{ll}
  \rho_f N_f=\rho N, \\[4pt]
    L_{f}=L,\\[4pt]
   N_f^2+2h(\tau_f)= N^2+2h(\tau),
  \end{array}
\right.
\end{equation}
where  $h(\tau)=\frac{\gamma \tau^{1-\gamma}}{\gamma-1}$ is the enthalpy.

We define the directional derivative
$$
\bar{\partial}_{s}=\cos\phi\partial_x+\sin\phi\partial_y \quad \mbox{along}\quad {\it S}_{\varepsilon}.
$$
From (\ref{Polar}) we have that along ${\it S}_{\varepsilon}$,
\begin{equation}\label{6609}
\begin{aligned}
&G_{u}\bar{\partial}_{s}u+G_v\bar{\partial}_{s}v-(\rho u-\rho_fu_f)\bar{\partial}_{s}u_f-
(\rho v-\rho_fv_f)\bar{\partial}_{s}v_f\\&\qquad\qquad\qquad -(u-u_f)\bar{\partial}_{s}(\rho_fu_f)-(v-v_f)\bar{\partial}_{s}(\rho_fv_f)=0,
\end{aligned}
\end{equation}
where
$$
G_u=\Big(\rho-\frac{\rho u^2}{c^2}\Big)(u-u_f)+(\rho u-\rho_fu_f)-v(v-v_f)\frac{\rho u}{c^2}
$$
and
$$
G_v=\Big(\rho-\frac{\rho v^2}{c^2}\Big)(v-v_f)+(\rho v-\rho_fv_f)-u(u-u_f)\frac{\rho v}{c^2}.
$$

Using (\ref{6603}), (\ref{52304}), (\ref{RH}) we have
\begin{equation}\label{6608}
\begin{aligned}
&(\rho u-\rho_fu_f)\bar{\partial}_{s}u_f+
(\rho v-\rho_fv_f)\bar{\partial}_{s}v_f\\=&
\big[(\rho v-\rho_fv_f)\cos\beta_f-(\rho u-\rho_fu_f)\sin\beta_f\big]\frac{q_f\bar{\partial}_{s}r_{+}^{f}}{2\cos A_f}
\\=&L_f(\rho-\rho_f)\sin(\phi-\beta_f)\frac{q_f\bar{\partial}_{s}r_{+}^{f}}{2\cos A_f},
\end{aligned}
\end{equation}
where $r_{+}^f(x, y)=\bar{r}_{+}(\eta)=\arctan\Big(\frac{\bar{v}(\eta)}{\bar{u}(\eta)}\Big)+\int_{q_{0}}^{\bar{q}(\eta)}
\frac{\sqrt{q^{2}-c^{2}}}{qc}{\rm d}q$, $A_f(x, y)=\bar{A}(\eta)$, and $q_f(x, y)=\bar{q}(\eta)$.

From (\ref{6603}) and (\ref{6606}) we have
\begin{equation}
\bar{\partial}_{s}(\rho_fu_f)=\left(-\frac{\rho_fu_f}{\sin(2A_f)}-\frac{\rho_fq_f\sin\beta_f}{2\cos A_f}\right)\bar{\partial}_{s}r_{+}^{f}
\end{equation}
and
\begin{equation}
\bar{\partial}_{s}(\rho_fv_f)=\left(-\frac{\rho_fv_f}{\sin(2A_f)}+\frac{\rho_fq_f\cos\beta_f}{2\cos A_f}\right)\bar{\partial}_{s}r_{+}^{f}.
\end{equation}
Thus, we get
\begin{equation}\label{6607}
\begin{aligned}
&(u-u_f)\bar{\partial}_{s}(\rho_fu_f)+(v-v_f)\bar{\partial}_{s}(\rho_fv_f)\\
~=~&(N-N_f)\Big(\sin\phi\bar{\partial}_{s}(\rho_fu_f)-\cos\phi\bar{\partial}_{s}(\rho_fv_f)\Big)\\~=~&
(N-N_f)\left(-\frac{\rho_fu_f\sin\phi}{\sin(2A_f)}-\frac{\rho_fq_f\sin\beta_f\sin\phi}{2\cos A_f}
+\frac{\rho_fv_f\cos\phi}{\sin(2A_f)}-\frac{\rho_fq_f\cos\beta_f\cos\phi}{2\cos A_f}\right)\bar{\partial}_{s}r_{-}^{f}
\\~=~&(N-N_f)\left(-\frac{\rho_fq_f\sin(\phi-\sigma_f)}{\sin(2A_f)}-\frac{\rho_fq_f\cos(\phi-\beta_f)}{2\cos A_f}\right)\bar{\partial}_{s}r_{+}^{f}.
\end{aligned}
\end{equation}

Combining (\ref{6609}), (\ref{6608}), and (\ref{6607}) we get
\begin{equation}\label{52402}
\begin{aligned}
G_u\bar{\partial}_{s}u+G_v\bar{\partial}_{s}v
=\mathcal{L}\bar{\partial}_{s}r_{+}^{f} \quad \mbox{along}\quad {\it S}_{\varepsilon},
\end{aligned}
\end{equation}
where
\begin{equation}\label{l}
\begin{aligned}
\mathcal{L}&=\frac{q_f(\rho-\rho_f)L\sin(\phi-\beta_f)}{2\cos A_f}-\rho_fq_f(N-N_f)\left(\frac{\sin(\phi-\sigma_f)}{\sin(2A_f)}+\frac{\cos(\phi-\beta_f)}{2\cos A_f}\right)
\\&=\frac{q_f(\rho-\rho_f)}{\sin(2A_f)}\Big(L\sin(\phi-\beta_f)\sin A_f+N\big(\sin(\phi-\sigma_f)+\cos(\phi-\beta_f)\sin A_f\big)\Big)\\&=
\frac{qq_f(\rho-\rho_f)}{\sin(2A_f)}\Big(\cos(\phi-\sigma)\sin(\phi-\beta_f)\sin A_f+\sin(\phi-\sigma)\big(\sin(\phi-\sigma_f)+\cos(\phi-\beta_f)\sin A_f\big)\Big)
\\&=
\frac{qq_f(\rho-\rho_f)}{\sin(2A_f)}\Big(\sin(\phi-\beta_f+\phi-\sigma)\sin A_f+\sin(\phi-\sigma)\sin(\phi-\sigma_f)\Big)\\&=
\frac{qq_f(\rho-\rho_f)}{\sin(2A_f)}\Big(\sin(\phi-\sigma_f+\phi-\sigma+A_f)\sin A_f+\sin(\phi-\sigma)\sin(\phi-\sigma_f)\Big)\\&=
\frac{qq_f(\rho-\rho_f)}{\sin(2A_f)}\Big(\sin(\phi-\sigma_f+\phi-\sigma)\cos A_f\sin A_f+\sin(\phi-\sigma)\sin(\phi-\sigma_f)\cos^2 A_f\\&\qquad\qquad\qquad\qquad+\cos(\phi-\sigma)\cos(\phi-\sigma_f)\sin^2 A_f\Big)>0,
\end{aligned}
\end{equation}
since $0<\phi-\sigma<\frac{\pi}{2}$, $0<A<\frac{\pi}{2}$, and $0<\phi-\sigma_f<\frac{\pi}{2}$.

We decompose $\bar{\partial}_{s}$ in two characteristic directions
\begin{equation}\label{72103}
\bar{\partial}_{s}=t_{+}\bar{\partial}_{+}+t_{-}\bar{\partial}_{-},
\end{equation}
where
$$
t_{+}=\frac{\sin(\phi-\beta )}{\sin(2A )}\quad \mbox{and}\quad t_{-}=\frac{\sin(\alpha -\phi)}{\sin(2A )}.
$$
Then by (\ref{11a}) and (\ref{72804a}) we have
\begin{equation}
\bar{\partial}_{s}u =\kappa (t_{+}\sin\beta \bar{\partial}_{+}c- t_{-}\sin\alpha \bar{\partial}_{-}c)
\quad \mbox{and}\quad  \bar{\partial}_{s}v =-\kappa (t_{+}\cos\beta \bar{\partial}_{+}c- t_{-}\cos\alpha \bar{\partial}_{-}c).
\end{equation}

The vector $(G_u, G_v)$ is actually a normal vector of the shock polar for the given front-side  state $(u_f, v_f)$ at the point $(u, v)$.
Let the angle $\chi$ be defined by
 $(\cos\chi, \sin\chi)=\Big(\frac{G_u}{\sqrt{G_u^2+G_v^2}}, \frac{G_v}{\sqrt{G_u^2+G_v^2}}\Big)$. Then, $\chi$ can also be seen as a function of $(u_f, v_f, u, v)$.
Hence, (\ref{52402}) can be written as
\begin{equation}\label{52505}
\kappa t_{+}\sin(\beta-\chi)\bar{\partial}_{+}c=-
\kappa t_{-}\sin(\chi-\alpha)\bar{\partial}_{-}c+\mathcal{G}^{-1}\mathcal{L}\bar{\partial}_{s}r_{+}^{f}\quad \mbox{along}\quad {\it S}_{\varepsilon},
\end{equation}
where $\mathcal{G}=\sqrt{G_u^2+G_v^2}$.

We define the following constants that related to the sonic shock:
\begin{equation}
\begin{aligned}
&\mathcal{L}_{*}=\mathcal{L}(u_0, 0, u_{*}, v_{*}), \quad \mathcal{G}_{*}=\sqrt{(G_u^2+G_v^2)(u_0, 0, u_{*}, v_{*})},\quad \phi_*=\arctan\Big(\frac{v_*}{u_0-u_*}\Big),\\&\alpha_*=\arctan\big(\frac{v_*}{u_*}\big)+\frac{\pi}{2},\quad
\beta_*=\arctan\big(\frac{v_*}{u_*}\big)-\frac{\pi}{2},\quad \chi_*=\chi(u_0, 0, u_{*}, v_{*}),\\& \ell_{*}=\mbox{arclength}(\overline{\mathrm{PD}}), \quad d_*= -\frac{4(\gamma-1)\mathcal{L}_*\sin(\phi_*+A_0)\cos A_0p'(\tau_0)}{\mathcal{G}_*\sin(\phi_*-\beta_*)\sin(\beta_*-\chi_*)\tau_0p''(\tau_0)\sqrt{x_{_\mathrm{P}}^2+(y_{_\mathrm{P}}-1)^2}}.
%\quad  a_{*}=\sin(\phi_{*}-\beta_{*})\sin(\beta_{*}-\chi_{*}), \\&\quad b_{*}=\sin(\alpha_{*}-\phi_{*})\sin(\chi_{*}-\alpha_{*}),  , \quad d_{*}=\frac{(\gamma-1)\mathcal{L}_{*}\sin(\phi_{*}+A_0)\sin A_0\cos A_0}{x_{_P}\sin(\phi_{*}-\beta_{*})\sin(\beta_{*}-\chi_{*})}\bar{r}_{+}'(-A_0).
\end{aligned}
\end{equation}

\begin{prop}
For the constants $\alpha_*$, $\beta_*$, $\phi_*$, and $\chi_*$ defined above, there hold the relations
\begin{equation}
\sin(\alpha_*-\phi_*)>0, \quad \sin(\phi_*-\beta_*)>0, \quad
\sin(\beta_*-\chi_*)<0,\quad \mbox{and}\quad \sin(\chi_*-\alpha_*)<0.
\end{equation}
\end{prop}
\begin{proof}
By the definitions of $\alpha_*$, $\beta_*$ and $\phi_*$ we have
 $$
\sin(\alpha_*-\phi_*)>0\quad \mbox{and} \quad \sin(\phi_*-\beta_*)>0.
 $$

In view of $u_*^2+v_*^2=c_*^2$, we have
$$
\begin{aligned}
&G_u(u_*, v_*, u_0, 0)u_*+G_v(u_*, v_*, u_0, 0)v_*\\
~=~&(\rho_*u_*-\rho_0u_0)u_*+\rho_*v_*^2=(\rho_*-\rho_0)u_0\cos^2\phi_*+\rho_*v_*^2>0.
\end{aligned}
$$
This implies
$$
\sin(\beta_*-\chi_*)<0\quad \mbox{and}\quad \sin(\chi_*-\alpha_*)<0;
$$
see Figure \ref{Figure2} (left).
This completes the proof.
\end{proof}

From (\ref{Polar}) we see that along the shock ${\it S}_{\varepsilon}$,  the variables $\phi$, $\alpha$, $\beta$, $\chi$, and $\mathcal{L}$ can be treated as algebraic functions of the upstream flow states $(r_{-}^{f}, r_{+}^{f})$ and the Mach angle $A$ of the downstream flow, i.e.,
$$
(\phi, \alpha, \beta, \chi, \mathcal{L})=\big(\hat{\phi}, \hat{\alpha}, \hat{\beta}, \hat{\chi}, \hat{\mathcal{L}}\big)\big(r_{-}^{f}, r_{+}^{f}, A\big)\quad \mbox{along}\quad {\it S}_{\varepsilon}.
$$
Therefore,
for any small $\epsilon>0$, there exists a small
  $\varrho>0$ such that when $\eta_0<\eta<\eta_0+\varrho$ and $\frac{\pi}{2}-\varrho<A<\frac{\pi}{2}$ then
\begin{equation}\label{91201a}
\Big|\big(\hat{\phi}, \hat{\alpha}, \hat{\beta}, \hat{\chi}, \hat{\mathcal{L}}\big)\big(\bar{r}_{-}(\eta), \bar{r}_{+}(\eta), A\big)-(\phi_{*}, \alpha_{*}, \beta_{*}, \chi_{*}, \hat{\mathcal{L}}_{*})\Big|<\epsilon.
\end{equation}
%where $(\phi, \alpha, \beta, \chi)=(\hat{\phi}, \hat{\alpha}, \hat{\beta}, \hat{\chi})(r_{-}^{f}, r_{+}^{f}, A)$.

%\subsection{Estimates for the first order derivatives}
We first consider (\ref{42501}) with the boundary  conditions (\ref{br2})--(\ref{br4}). The problem (\ref{42501}), (\ref{br1})--(\ref{br3}) is a typical free boundary problem. By the result in Li and Yu \cite{Li-Yu} (Chapter 3.2), we know that this problem admits a local solution. Moreover, by Lemma \ref{simple} we know that the solution is a simple wave with straight $C_{-}$ characteristic lines issued from ${\it S}_{\varepsilon}$.
Since the centered simple wave issued from $B$ is a rarefaction simple wave, by (\ref{6606}) we have
$$
\bar{\partial}_{s}r_{+}^{f}>0\quad \mbox{along}\quad {\it S}_{\varepsilon}.
$$
Combining this with (\ref{l}) and (\ref{52505}), we have that the simple wave  satisfies
$$
\bar{\partial}_{+}c<0.
$$
Combining this with (\ref{7a}), (\ref{8a}), and (\ref{72103}), we find that the solution satisfies
$$
\bar{\partial}_{s}\beta>0\quad \mbox{along}\quad {\it S}_{\varepsilon}.
$$
So, the simple wave is a rarefaction simple wave, and the straight characteristic lines emanating from ${\it S}_{\varepsilon}$ will not intersect each other before reaching the solid wall.
From the point $\mathrm{D}_{\varepsilon}$, we draw a $C_{+}$ cross characteristic curve of the simple wave. When $\varepsilon$ is sufficiently small, this characteristic curve intersects ${\it S}_{\varepsilon}$ at a point $\mathrm{E}_{\varepsilon}(x_{_{\mathrm{E}_{\varepsilon}}}, y_{_{\mathrm{E}_{\varepsilon}}})$.
So, the free boundary  problem (\ref{42501}), (\ref{br1})--(\ref{br3}) admits a rarefaction simple wave solution on a curvilinear triangle domain $\Delta_{\varepsilon}$ bounded by $\overline{\mathrm{P}_{\varepsilon}\mathrm{D}_{\varepsilon}}$, $\wideparen{\mathrm{D}_{\varepsilon}\mathrm{E}_{\varepsilon}}$, and $\wideparen{\mathrm{P}_{\varepsilon}\mathrm{E}_{\varepsilon}}$.

%=============
\begin{figure}[htbp]
\begin{center}
\includegraphics[scale=0.35]{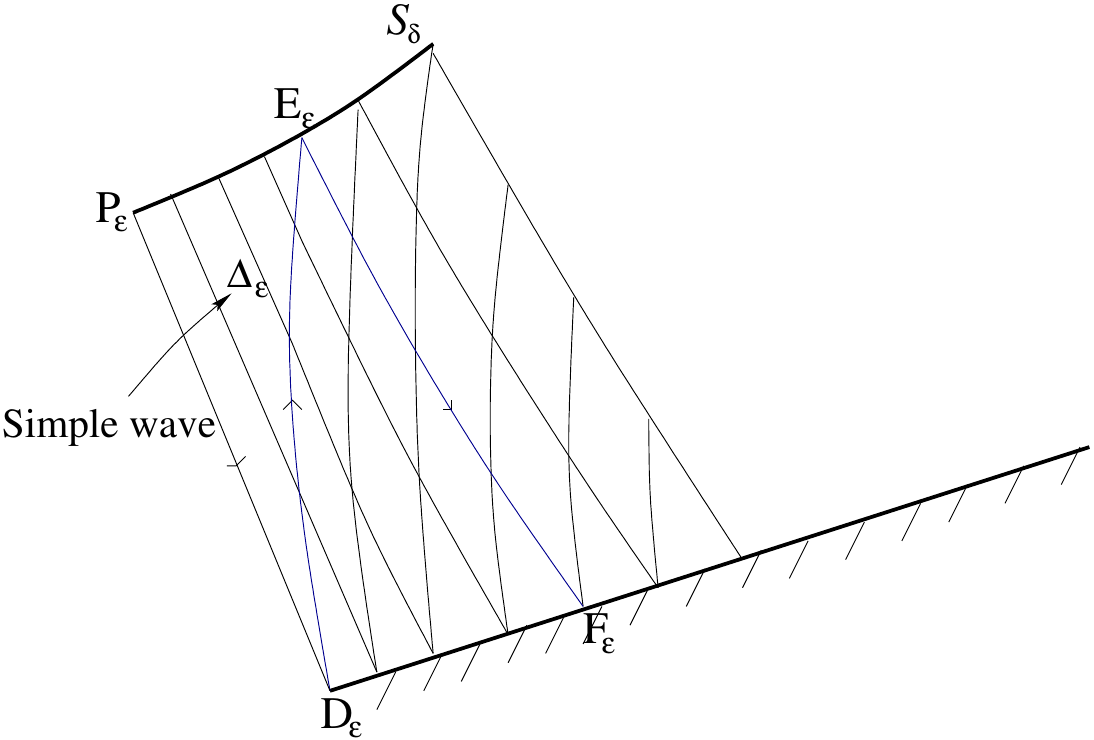}
\caption{\footnotesize Simple wave issued from the shock.}
\label{Figure5}
\end{center}
\end{figure}
%==========

\begin{lem}\label{lem3}
When $\varepsilon$ is sufficiently small, $\eta_0<\arctan\big(\frac{y_{_{\mathrm{E}_{\varepsilon}}}-1}{x_{_{\mathrm{E}_{\varepsilon}}}}\big)<\eta_1$. Moreover,
$$
(x_{_{\mathrm{E}_{\varepsilon}}}, y_{_{\mathrm{E}_{\varepsilon}}})\rightarrow (x_{_\mathrm{P}}, y_{_\mathrm{P}}) \quad \mbox{as}\quad \varepsilon\rightarrow 0.
$$
\end{lem}

\begin{proof}
For convenience, we represent the shock by the following one-parametric form
$$
x=x_{s}(\eta), \quad y=y_{s}(\eta),\quad  \eta_0<\eta<\eta_*,
$$
where $\eta_*=\arctan\big(\frac{y_{_{\mathrm{E}_{\varepsilon}}}-1}{x_{_{\mathrm{E}_{\varepsilon}}}}\big)$.
The states on its backside are characterized by
$$
u=u_b(\eta), \quad v=v_b(\eta), \quad c=c_b(\eta), \quad \beta=\beta_b(\eta),\quad \alpha=\alpha_b(\eta),\quad \mbox{and}  \quad A=A_b(\eta).
$$
Therefore, the straight characteristic curves of the simple wave can be represented by
$$
x=x_s(\eta)+\ell\cos\beta_b(\eta), \quad y= y_s(\eta)+\ell\sin\beta_b(\eta), \quad \ell>0.
$$

We denote the $C_{+}$ cross characteristic curve issued from the point $\mathrm{D}_{\varepsilon}$ by the following parametric form
with $\eta$ as the parameter:
$$
x=x_{+}(\eta):=x_s(\eta)+\ell_{+}(\eta)\cos\beta_b(\eta), \quad y=y_{+}(\eta):= y_s(\eta)+\ell_{+}(\eta)\sin\beta_b(\eta).
$$
Combining this with $y_{+}'(\eta)=x_{+}'(\eta)\tan\alpha_b(\eta)$, we have
\begin{equation}\label{82401a}
\sin(2A_b(\eta))r_{+}'(\eta)=\ell_{+}(\eta)\beta_b'(\eta)\cos(2A_b(\eta))+\cos(\alpha_b(\eta))y_s'(\eta)-
\sin(\alpha_b(\eta))x_s'(\eta).
\end{equation}

By a direct computation, we also have
$$
\cos(\alpha_b(\eta))y_s'(\eta)-
\sin(\alpha_b(\eta))x_s'(\eta)=\frac{\sin(\phi(\eta)-\alpha_b(\eta))}{\sqrt{(x_s'(\eta))^2+(y_s'(\eta))^2}}<0,
$$
since $0<\alpha_b(\eta)-\phi(\eta)<\pi$ along the shock. So, by (\ref{82401a}) we have
\begin{equation}\label{82402a}
\sin(2A_b(\eta))\ell_{+}'(\eta)<\ell_{+}(\eta)\beta_b'(\eta)\cos(2A_b(\eta)).
\end{equation}

From (\ref{7a}) and (\ref{72103}) we have
$$
\beta_b'(\eta)=-(1+\kappa)\tan(A_b(\eta))c_b^{-1}(\eta)c_b'(\eta).
$$
From Bernoulli's law (\ref{5802}) we have
\begin{equation}\label{82405a}
c_b'(\eta)=\frac{c_b(\eta)\cot(A_b(\eta))A_b'(\eta)}{1+\kappa \sin^2(A_b(\eta))}.
\end{equation}
Thus, by (\ref{82402a}) we have
\begin{equation}\label{82403a}
\frac{\ell_{+}'(\eta)}{\ell_{+}(\eta)}<-\frac{(1+\kappa)\cos(2A_b(\eta))A_b'(\eta)}{(1+\kappa \sin^2(A_b(\eta)))\sin(2A_b(\eta))}.
\end{equation}

From (\ref{72103}) and (\ref{52505}) we have
\begin{equation}\label{91801a}
\kappa \sin(\beta_b(\eta)-\chi(\eta))c_b'(\eta)=\frac{\mathcal{L}(\eta)\bar{r}_{+}'(\eta)}{\mathcal{G}(\eta)},
\end{equation}
where $\mathcal{G}(\eta)=\sqrt{(G_u+
G_v)^2(\bar{u}(\eta), \bar{v}(\eta), u_b(\eta), v_b(\eta))}$, $\chi(\eta)=\chi(\bar{u}(\eta), \bar{v}(\eta), u_b(\eta), v_b(\eta))$, and $\mathcal{L}(\eta)$ $= \mathcal{L}(\bar{u}(\eta), \bar{v}(\eta), u_b(\eta), v_b(\eta))$.
Inserting (\ref{82405a}) into (\ref{91801a}), we obtain
\begin{equation}\label{82406a}
\frac{\kappa \sin(\beta_b(\eta)-\chi(\eta))c_b(\eta)\cot(A_b(\eta))A_b'(\eta)}{1+\kappa \sin^2(A_b(\eta))}=\frac{\mathcal{L}(\eta)\bar{r}_{+}'(\eta)}{\mathcal{G}(\eta)}.
\end{equation}

From (\ref{102304a}) we have
\begin{equation}\label{82501a}
\bar{r}_{+}'(\eta)=-\frac{4 p'(\bar{\tau})\cos^2 \bar{A}}{\bar{\tau}p''(\bar{\tau})}.
\end{equation}

Combining (\ref{82403a})--(\ref{82501a}), we have
$$
\eta_*-\eta_0\rightarrow 0\quad \mbox{as}\quad \varepsilon\rightarrow 0,
$$
since $A_b(\eta_0)\rightarrow \frac{\pi}{2}$ as $\varepsilon\rightarrow 0$.
This completes the proof of the lemma.
\end{proof}

%We write (\ref{52505}) into the following form
%\begin{equation}\label{52505A}
%\sin(\phi-\beta)\sin(\beta-\chi)\frac{\bar{\partial}_{+}c}{\cos A}=-
%\sin(\alpha-\phi)\sin(\chi-\alpha)\frac{\bar{\partial}_{-}c}{\cos A}+\frac{(\gamma-1)\mathcal{L}\sin A\sin(\phi-\beta_f)}{\sin (2A_f)}\bar{\partial}_{+}r_{+}^{f},
%\end{equation}

%A direct computation yields
%\begin{equation}
%\bar{\partial}_{+}r_{+}^{f}=\frac{2\sin A_0\bar{r}_{+}'(-A_0)}{x_{_P}(1+\tan^2 A_0)}.
%\end{equation}
\begin{lem}\label{lem32}
Assume that $\varepsilon$ is sufficiently small. Then there is a small $\delta>0$, independent of $\varepsilon$, such that if the free boundary problem (\ref{42501}), (\ref{br1})--(\ref{br4}) admits a continuous and piecewise-smooth solution on a domain  $\Sigma_{\varepsilon}(\delta')$ bounded by $\overline{\mathrm{P_{\varepsilon}D_{\varepsilon}}}$, $y=\psi_{\varepsilon}(x)$ ($x_{_\mathrm{P_{\varepsilon}}}<x<x_{_\mathrm{P_{\varepsilon}}}+\delta'$), the solid wall $y=x\tan\theta_w$, and a forward $C_{-}$ characteristic line issued from $(x_{_\mathrm{P_{\varepsilon}}}+\delta', \psi_{\varepsilon}(x_{_\mathrm{P_{\varepsilon}}}+\delta'))$ for some $\delta'<\delta$, then the solution satisfies
\begin{equation}\label{estimate}
2d_{*}<\frac{\bar{\partial}_{+}c}{\cos A}<0
 \quad \mbox{and}\quad -\frac{4(\gamma-1)c_*}{\ell_{*}(\gamma+1)}<\frac{\bar{\partial}_{-}c}{\cos^2 A}\leq 0
 \quad \mbox{on}\quad\Sigma_{\varepsilon}(\delta').
\end{equation}
\end{lem}
\begin{proof}
From the theory of characteristics, we know that the solution has weak discontinuities along the characteristic curve emanating from the point $\mathrm{D}_{\varepsilon}$, as well as along the subsequent characteristic curves reflected from $S_{\varepsilon}$ and from the lower wall; i.e.,  $\bar{\partial}_{+}c$ or $\bar{\partial}_{-}c$ has a jump discontinuity  on these curves.

The proof of this lemma proceeds in three steps.

\noindent
{\it Step 1.} We first show that for any fixed $\varrho>0$,
if (\ref{estimate}) holds and $\varepsilon$ and $\delta$ are sufficiently small,  then
\begin{equation}\label{71906}
\frac{\pi}{2}-\varrho<A<\frac{\pi}{2}\quad \mbox{on} \quad  {\it S_{\varepsilon}}.
\end{equation}

From (\ref{72103}), we have
\begin{equation}
\bar{\partial}_{s}c=\frac{\sin(\phi-\beta)}{2\sin A}\cdot\frac{\bar{\partial}_{+}c}{\cos A}+\frac{\sin(\alpha-\phi)}{2\sin A}\cdot\frac{\bar{\partial}_{-}c}{\cos A}\quad \mbox{along}\quad {\it S_{\varepsilon}}.
\end{equation}
Combining this with Bernoulli's law (\ref{5802}), we get
\begin{equation}\label{81301}
\cos A\bar{\partial}_{s}A=\frac{1+\kappa \sin^2 A}{q}\left(\frac{\sin(\phi-\beta)}{2\sin A}\cdot\frac{\bar{\partial}_{+}c}{\cos A}+\frac{\sin(\alpha-\phi)}{2\sin A}\cdot\frac{\bar{\partial}_{-}c}{\cos A}\right)\quad \mbox{along}\quad {\it S_{\varepsilon}}.
\end{equation}
Then by (\ref{91201a}) we get (\ref{71906}).

\vskip 4pt
\noindent
{\it Step 2}.
From (\ref{52505}),
we have
$$
\begin{aligned}
\frac{\bar{\partial}_{+}c}{\cos A}&=
\frac{\sin(\phi-\alpha)\sin(\chi-\alpha)\cos A}{\sin(\phi-\beta)\sin(\beta-\chi)}\cdot\frac{\bar{\partial}_{-}c}{\cos^2 A}+\frac{(\gamma-1)\mathcal{L}\sin A\sin(\phi-\beta_f)\bar{\partial}_{+}r_{+}^{f}}{\mathcal{G}\sin(\phi-\beta)\sin(\beta-\chi)\sin (2A_f)}
 \\&=\frac{\sin(\phi-\alpha)\sin(\chi-\alpha)\cos A}{\sin(\phi-\beta)\sin(\beta-\chi)}\cdot\frac{\bar{\partial}_{-}c}{\cos^2 A}+\frac{(\gamma-1)\mathcal{L}\sin A\sin(\phi-\beta_f)\bar{r}_{+}'(\eta)}{\mathcal{G}\sin(\phi-\beta)\sin(\beta-\chi)\sqrt{x^2+(y-1)^2}}
 \\&=\frac{\sin(\phi-\alpha)\sin(\chi-\alpha)\cos A}{\sin(\phi-\beta)\sin(\beta-\chi)}\cdot\frac{\bar{\partial}_{-}c}{\cos^2 A} -\frac{4(\gamma-1)\mathcal{L}\sin A\sin(\phi-\beta_f)\cos A_fp'(\tau_f)}{\mathcal{G}\sin(\phi-\beta)\sin(\beta-\chi)\tau_fp''(\tau_f)\sqrt{x^2+(y-1)^2}}
\end{aligned}
$$
along $S_{\varepsilon}$.
Combining this with (\ref{81301}) we know that if $-\frac{4(\gamma-1)c_*}{\ell_{*}(\gamma+1)}\leq \frac{\bar{\partial}_{-}c}{\cos^2 A}\leq 0$ along ${\it S_{\varepsilon}}$, then when $\varepsilon$ and $\delta$ are sufficiently small,  we have
\begin{equation}\label{71910}
2d_{*}<\frac{\bar{\partial}_{+}c}{\cos A}<\frac{d_{*}}{2}  \quad \mbox{along}\quad {\it S_{\varepsilon}}.
\end{equation}

By a direct computation, we have
$$
\bar{\partial}_{-}\left(\frac{\bar{\partial}_{+}c}{\cos A}\right)= \frac{(\gamma+1)\bar{\partial}_{+}c\bar{\partial}_{+}c}{2c(\gamma-1)\cos^3A}+\frac{(\gamma+1)-2\sin^2 2A}{2c(\gamma-1)\cos^3 A}
   \bar{\partial}_{-}c\bar{\partial}_{+}c+\frac{\sin A\bar{\partial}_{+}c\bar{\partial}_{-}A}{\cos^2 A}\geq 0
$$
provided that $\bar{\partial}_{\pm}c\leq 0$.
Hence, by (\ref{71910}) we have
\begin{equation}
2d_{*}<\frac{\bar{\partial}_{+}c}{\cos A}<0\quad  \mbox{in}\quad \Sigma_{\varepsilon}(\delta'),
\end{equation}
provided that $\bar{\partial}_{-}c\leq 0$ in $\Sigma_{\varepsilon}(\delta')$.

From the second equation of (\ref{wcd}) we also have
\begin{equation}\label{81302}
\bar{\partial}_{-}\Big(\frac{\cos^2 A}{\bar{\partial}_{+}c}\Big)~=~-\frac{\gamma+1}{2c(\gamma-1)} -
\left( \frac{\mathcal{F}\bar{\partial}_{-}c}{c\cos^2 A}\right)\frac{\cos^2 A}{\bar{\partial}_{+}c}<-\frac{\gamma+1}{2c(\gamma-1)}\quad  \mbox{in}\quad \Sigma_{\varepsilon}(\delta'),
\end{equation}
provided that $\bar{\partial}_{-}c\leq 0$ in $\Sigma_{\varepsilon}(\delta')$.

For convenience, we use $\mathcal{W}$ to denote the lower wall.
Integrating (\ref{81302}) along the characteristic curves issued from ${\it S_{\varepsilon}}$, we get that when $\varepsilon$ and $\delta$ are sufficiently small,
\begin{equation}\label{71907}
\frac{\bar{\partial}_{+}c}{\cos^2 A}>-\frac{4(\gamma-1)c_*}{\ell_{*}(\gamma+1)}\quad\mbox{on}\quad \mathcal{W}\cap\overline{\Sigma_{\varepsilon}(\delta')},
\end{equation}
provided that $\bar{\partial}_{-}c\leq 0$ in $\Sigma_{\varepsilon}(\delta')$.

\vskip 2pt
\noindent
{\it Step 3.}
From (\ref{br4}) we have
$$
\bar{\partial}_0\sigma=0\quad \mbox{along}\quad \mathcal{W}\cap\overline{\Sigma_{\varepsilon}(\delta')}.
$$
Since $\bar{\partial}_0\sigma=\frac{1}{4}(\bar{\partial}_{+}+\bar{\partial}_{-})(\alpha+\beta)$, by (\ref{7a})--(\ref{8a}) we have
\begin{equation}\label{71905}
\bar{\partial}_{+}c=\bar{\partial}_{-}c\quad  \mbox{along}\quad \mathcal{W}\cap\overline{\Sigma_{\varepsilon}(\delta')}.
\end{equation}
Combining this with (\ref{71907}) we have
\begin{equation}\label{71902}
-\frac{4(\gamma-1)c_*}{(\gamma+1)\ell_{*}}<\frac{\bar{\partial}_{-}c}{\cos^2 A}<0\quad\mbox{on}\quad \mathcal{W}\cap\overline{\Sigma_{\varepsilon}(\delta')}.
\end{equation}
Furthermore, by the second equation of (\ref{wcd}) we have
$$
-\frac{4(\gamma-1)c_*}{(\gamma+1)\ell_{*}}<\frac{\bar{\partial}_{-}c}{\cos^2 A}<0\quad\mbox{on}\quad \Sigma_{\varepsilon}(\delta')\setminus \Delta_{\varepsilon}.
$$

Therefore, by the continuity argument we get (\ref{estimate}).
This completes the proof.
\end{proof}

%From Lemma \ref{lem32}
%we have the following conclusion.
\begin{lem}
Assume that $\delta$ is sufficiently small. Then, for any  small $\varepsilon>0$, the regularized problem (\ref{42501}), (\ref{br1})--(\ref{br4}) admits a classical solution in $\Sigma_{\varepsilon}(\delta)$. Moreover, there exists a positive constant $\mathcal{M}$ independent of $\varepsilon$ such that the solution satisfies
\begin{equation}\label{2401}
|Du|+|Dv|<\mathcal{M} \quad \mbox{on}\quad \Sigma_{\varepsilon}(\delta).
\end{equation}
\end{lem}
\begin{proof}
The existence follows from Lemma \ref{lem32} and the method of characteristics. The estimate (\ref{2401}) can be obtained by using (\ref{11a}), (\ref{72804a}),
$$
\partial_{x}=-\frac{\sin\beta\bar{\partial}_{+}-\sin\alpha\bar{\partial}_{-}}{\sin2A} \quad
\mbox{and}\quad
\partial_{y}=\frac{\cos\beta\bar{\partial}_{+}-\cos\alpha\bar{\partial}_{-}}{\sin2A}.
$$
Since the computations are direct, we omit the details.
\end{proof}

From the second equation of (\ref{wcd}) and the second estimate in (\ref{estimate}), we know that there is a constant $\mathcal{C}<0$ independent of $\varepsilon$ and $\delta$ such that
\begin{equation}\label{92401a}
\bar{\partial}_{-}\beta>0, \quad \bar{\partial}_{+}\alpha<0,\quad \mbox{and}\quad\frac{\bar{\partial}_{+}c}{\cos ^2 A}<\mathcal{C}\quad \mbox{on} \quad\Sigma_{\varepsilon}(\delta).
\end{equation}
Combining this with (\ref{pma}) we obtain
\begin{equation}\label{92401b}
\bar{\partial}_{0}A<\frac{\bar{\partial}_{+}A+\bar{\partial}_{-}A}{2\cos A}
=\frac{(1+\kappa\sin^2 A)(\bar{\partial}_{+}c+\bar{\partial}_{-}c)}{2c\cot A\cos A}
<\frac{\sin A(1+\kappa\sin^2 A) \mathcal{C}}{2c}\quad \mbox{on}\quad\Sigma_{\varepsilon}(\delta).
\end{equation}
From (\ref{7a}), (\ref{8a}),  (\ref{72103}), (\ref{estimate}), and (\ref{71910}),  we have
\begin{equation}\label{92401c}
\bar{\partial}_{s}\beta>0\quad \mbox{on}\quad{S_{\varepsilon}}.
\end{equation}
From (\ref{92401a})--(\ref{92401c}) we can see that for a fixed small $\delta>0$, the areas of $\Sigma_{\varepsilon}(\delta)$ are uniformly bounded below away from zero with respect to $\varepsilon$.

Therefore, by the Ascoli-Arzel\`{a} theorem and a standard diagonal procedure we obtain a Lipschitz-continuous sonic-supersonic solution to (\ref{42501}), (\ref{b1})--(\ref{b4}) on a quadrilateral domain $\Sigma(\delta)$ bounded by $\overline{\mathrm{PD}}$, $y=\psi(x)$ ($x_{_\mathrm{P}}<x<x_{_\mathrm{P}}+\delta$), the lower wall $y=x\tan\theta_*$, and a forward $C_{-}$ characteristic line issued from the point $(x_{_\mathrm{P}}+\delta, \psi(x_{_\mathrm{P}}+\delta))$.
 This completes the proof of Theorem \ref{main}.

\vskip 32pt

\section*{Acknowledgement}
This work was partially supported by National Natural Science Foundation of China (No. 12671262) and Natural Science Foundation of Shanghai (No. 23ZR1422100).

\vskip 48pt
%%%%%%%%%%%%%%%%%%%%%%%%%%%%%%%%%%%%%%%%%%%%%%%%%%%%%%%%%%%%%%%%%%%%%%%%%%%%%%%%%%%%%%%%

\end{document}